\documentclass[10pt]{amsart}
\usepackage{amssymb,amsmath,amsthm,hyperref}
\usepackage[height=2em,width=2em]{diagrams}

\DeclareMathOperator{\cl}{cl}   % closure operator

\DeclareMathOperator{\rk}{rk}

\DeclareMathOperator{\acl}{acl}   % algebraic closure operator
\DeclareMathOperator{\td}{td}  %transcendence degree
\DeclareMathOperator{\ldim}{ldim}  %linear dimension
\DeclareMathOperator{\Ann}{Ann}   % Annihilator

\DeclareMathOperator{\Der}{Der}  %space of Derivations
\newcommand{\D}{\ensuremath{\partial}} %derivation

\DeclareMathOperator{\ecl}{ecl} % exponential algebraic closure

\newcommand{\card}[1]{\left| #1 \right|}   %cardinality

\newcommand{\restrict}[1]{\ensuremath{\!\!\upharpoonright_{#1}}}

\newcommand{\N}{\ensuremath{\mathbb{N}}}
\newcommand{\Z}{\ensuremath{\mathbb{Z}}}
\newcommand{\Q}{\ensuremath{\mathbb{Q}}}
\newcommand{\R}{\ensuremath{\mathbb{R}}}
\newcommand{\C}{\ensuremath{\mathbb{C}}}
\newcommand{\Rexp}{\ensuremath{\mathbb{R}_{\mathrm{exp}}}}
\newcommand{\Cexp}{\ensuremath{\mathbb{C}_{\mathrm{exp}}}}

\newcommand{\proves}{\ensuremath{\vdash}}

\newcommand{\ga}{\ensuremath{{\mathbb{G}_\mathrm{a}}}}   %additive group of a field
\newcommand{\gm}{\ensuremath{{\mathbb{G}_\mathrm{m}}}}  %mult group of a field
\renewcommand{\phi}{\varphi}
\renewcommand{\le}{\ensuremath{\leqslant}}
\renewcommand{\ge}{\ensuremath{\geqslant}}
\newcommand{\tuple}[1]{\ensuremath{\langle #1 \rangle}}
\newcommand{\class}[2]{\ensuremath{\left\{ #1 \,\left|\, #2 \right.\right\}}}

\newcommand{\into}{\hookrightarrow}

\newcommand{\subs}{\subseteq} % diagrams package uses subset for hook
\newcommand{\sups}{\supseteq} % superset
\newcommand{\minus}{\ensuremath{\smallsetminus}}
\newcommand{\strong}{\ensuremath{\lhd}} % strong embedding - better
\newcommand{\nstrong}{\ensuremath{\not\kern-4pt\lhd\;}} % nonstrong embedding

\newbox\noforkbox \newdimen\forklinewidth
\forklinewidth=0.3pt \setbox0\hbox{$\textstyle\smile$}
\setbox1\hbox to \wd0{\hfil\vrule width \forklinewidth depth-2pt
  height 10pt \hfil}
\wd1=0 cm \setbox\noforkbox\hbox{\lower 2pt\box1\lower
2pt\box0\relax}
\def\unionstick{\mathop{\copy\noforkbox}\limits}

\def\nonfork_#1{\unionstick_{\textstyle #1}}

\setbox0\hbox{$\textstyle\smile$} \setbox1\hbox to \wd0{\hfil{\sl
/\/}\hfil} \setbox2\hbox to \wd0{\hfil\vrule height 10pt depth
-2pt width
                \forklinewidth\hfil}
\newbox\doesforkbox
\setbox\doesforkbox\hbox{\lower 2pt\box1 \lower
2pt\box2\lower2pt\box0\relax}
\def\nunionstick{\mathop{\copy\doesforkbox}\limits}

\def\fork_#1{\nunionstick_{\textstyle #1}}

\newcommand{\ra}[3]{\ensuremath{#1 \stackrel{#2}{\longrightarrow} #3}}
\newcommand{\leteq}{\mathrel{\mathop:}=}

\newtheorem{prop}{Proposition}[section]
\newtheorem{cor}[prop]{Corollary}
\newtheorem{theorem}[prop]{Theorem}
\newtheorem{lemma}[prop]{Lemma}

\newtheorem{fact}[prop]{Fact}

\theoremstyle{definition}
\newtheorem{defn}[prop]{Definition}
\newtheorem{example}[prop]{Example}

\newtheorem{remark}[prop]{Remark}

\DeclareMathOperator{\EDer}{EDer}
\newcommand{\egg}{exponential-graph-generated}

\title{Exponential Algebraicity in Exponential Fields}
\author{Jonathan Kirby}
\address{Mathematical Institute\\University of Oxford\\UK}
\subjclass[2000]{03C60}
\begin{document}

\begin{abstract}
  I give an algebraic proof that the exponential algebraic closure operator in an exponential
  field is always a pregeometry, and show that its dimension function satisfies a weak Schanuel property. A corollary is that there are at most countably many essential counterexamples to Schanuel's conjecture.
\end{abstract}

\maketitle
%\tableofcontents

\section{Introduction}

In a field, the notion of algebraicity is captured by the algebraic closure operator, $\acl$. Algebraic closure is a pregeometry, that is, a closure operator of finite character satisfying the Steinitz exchange property
\[a \in \acl(C \cup\{b\}) \minus \acl(C) \implies b \in \acl(C \cup\{a\})\]
hence it gives rise to a dimension function, in this case transcendence degree. The analogous closure operator, $\ecl^F$, in an exponential field $F$ was defined by Macintyre \cite{Macintyre96}. In the special case of the real exponential field $\Rexp = \tuple{\R;+,\cdot,\exp}$, where $\exp$ is the usual exponential function $x \mapsto e^x$, Wilkie showed that $\ecl^\R$ is a pregeometry. His technique was to define a pregeometry $\cl^\R$ by derivations, and, using techniques of o-minimality and real analysis, to construct enough derivations to show that the two closure operators were equal. He later extended the result to the complex exponential field $\Cexp$ \cite{Wilkie_local_definability}, still using analytic techniques and the major theorem that the real field with exponentiation and restricted analytic functions is o-minimal.

Looking to study \Cexp\ in another way, Zilber \cite{Zilber05peACF0} constructed an exponential
field using the \emph{amalgamation of strong extensions} technique of Hrushovski \cite{Hru93}, and conjectured that it is isomorphic to \Cexp. His exponential field comes with a pregeometry satisfying an important transcendence property, the Schanuel property.

In this paper I give an algebraic proof of the generalization of Wilkie's result to an arbitrary exponential field:
\begin{theorem}\label{ecl is pregeometry}
  For any (total or partial) exponential field $F$, the closure operator $\ecl^F$ is a pregeometry, and it always agrees with the pregeometry $\cl^F$ defined using derivations.
\end{theorem}
 Furthermore, in every exponential field $F$, the dimension function $\dim^F$ associated with $\ecl^F$ satisfies a weak form of the Schanuel property:
\begin{theorem}\label{Schanuel property}
  Suppose $C \subs F$ is $\ecl^F$-closed. Let $x_1,\ldots,x_n \in F$. Then
  \[\delta(\bar{x}/C) \leteq \td(\bar{x},\exp(\bar{x})/C) - \ldim_\Q(\bar{x}/C) \ge \dim^F(\bar{x}/C).\] 
\end{theorem}
For any subsets $X,Y$ of a field $F$ (of characteristic zero), $\td(X/Y)$ means the transcendence
degree of the field extension $\Q(X,Y)/\Q(Y)$. Writing $\langle X \rangle_\Q$ for the $\Q$-linear span of $X$, $\ldim_\Q(X/Y)$ means the dimension of the quotient $\Q$-vector space $\langle X,Y
\rangle_\Q / \langle Y \rangle_\Q$.

In Hrushovski's constructions, the \emph{predimension function} $\delta$ characterises the dimension function. In this case $\delta$ does not directly give information about $\ecl^F(\emptyset)$, but $\delta$ and $\ecl^F(\emptyset)$ together determine the dimension function:
\begin{theorem}\label{dim characterization}
$\dim^F(\bar{x}) = \min \class{\delta(\bar{x}\bar{y}/\ecl^F(\emptyset))}{\bar{y} \subs F}$
\end{theorem}
The full Schanuel property states that $\delta(\bar{x}) \ge 0$ for all $\bar{x}$, and under this condition we can replace $\ecl^F(\emptyset)$ by $\emptyset$ in the above theorems. In the complex case this is Schanuel's conjecture, which is considered out of reach. However, we can show:
\begin{theorem}\label{SC counterexamples}
  There are at most countably many essential counterexamples to Schanuel's conjecture.
\end{theorem}
The notion of an essential counterexample must be explained. A counterexample to Schanuel's conjecture is a tuple $\bar{a} = (a_1,\ldots,a_n)$ of complex numbers such that $\delta(\bar{a}) < 0$. If there exists $\bar{a}$ such that $\delta(\bar{a}) < -1$ then for any $b \in \C$, $\delta(\bar{a}b) \le \delta(\bar{a}) + 1 < 0$, so there would be continuum-many counterexamples. However, if $\delta(\bar{a}b) = \delta(\bar{a}) + 1$ then $b$ is not contributing to the counterexample, so we want to exclude such cases. Note also that the value of $\delta(\bar{a})$ depends only on the $\Q$-linear span of $\bar{a}$. We define an \emph{essential counterexample} to be a counterexample $\bar{a}$ such that $\delta(\bar{a}) \le \delta(\bar{c})$ for any tuple $\bar{c}$ from the $\Q$-span of $\bar{a}$. Thus every counterexample contains an essential counterexample in its $\Q$-linear span.

To prove these theorems we construct derivations on exponential fields and show they can be extended to strong extensions of these fields. This seems to be a very non-trivial fact, depending on a theorem of Ax \cite{Ax71}. The techniques in this paper can probably
be extended to any collection of functions for which a similar result
is known. In particular, they should work for fields with formal
analogues of the Weierstrass $\wp$-functions, and the exponential maps
of other semiabelian varieties, using the analogues of Ax's theorem
given in \cite{TEDESV}.

\section{Exponential rings and fields}

 In this paper, a \emph{ring} $R =
\tuple{R;+,\cdot}$ is always commutative, with $1$. We write $\ga(R)$
for the additive group $\tuple{R;+}$ and $\gm(R)$ for the
multiplicative group $\tuple{R^\times;\cdot}$ of units of $R$.

\begin{defn}
  An \emph{exponential ring} (or \emph{E-ring}) is a ring $R$ equipped
  with a homomorphism $\exp_R$ (also written $\exp$, or $x \mapsto
  e^x$) from $\ga(R)$ to $\gm(R)$.

  We adopt the convention that an \emph{E-field} is an E-ring which is
  a field \emph{of characteristic zero}. Furthermore an
  \emph{E-domain} is an E-ring with no zero divisors which is also a
  $\Q$-algebra.
\end{defn}
Note that if $R$ is an E-ring of positive characteristic $p$ (that is,
$p$ is the least non-zero natural number such that
$\underbrace{1+\cdots+1}_p = 0$), then for each $x \in R$, $(e^x)^p =
e^0 = 1$. In particular, if $R$ is a domain then $p$ is prime and $0 =
(e^x)^p - 1 = (e^x - 1)^p$, so the exponential map is trivial. This is
the reason for the convention that E-domains and E-fields are always
of characteristic zero. It will be convenient for defining strong
embeddings later to insist that E-domains are \Q-algebras.

We will also need the notion of a \emph{partial E-domain}, where the exponential map is defined only on a subgroup of $\ga(R)$. To have the most useful notion of embedding, we give the formal definition as a two-sorted structure.
\begin{defn}
  A \emph{partial E-domain} is a two-sorted structure 
\[\tuple{R,A(R);+_R,\cdot,+_A,(q\cdot)_{q \in \Q},\alpha,\exp_R}\]
 where $\tuple{R;+_R,\cdot}$ is a domain, $\tuple{A(R);+_A,(q\cdot)_{q \in \Q}}$ is a \Q-vector space, $\ra{\tuple{A(R);+_A}}{\alpha}{\tuple{R;+_R}}$ is an injective homomorphism of additive groups, and $\ra{\tuple{A(R);+_A}}{\exp_R}{\tuple{R;\cdot}}$ is a homomorphism. We identify $A(R)$ with its image under $\alpha$, and write $+$ for both $+_A$ and $+_R$.
\end{defn}

We take the natural definitions of homomorphisms and embeddings of E-rings and partial E-domains. Thus a homomorphism of E-rings $\ra{R}{\phi}{S}$ is a ring homomorphism which preserves the exponential map. A homomorphism of partial E-domains is a ring homomorphism such that for each $x \in A(R)$, we have $\phi(x) \in A(S)$ and $\exp_S(\phi(x)) = \phi(\exp_R(x))$. The two-sorted definition of partial E-domains means that in an embedding $R \into S$, it is possible to have an element $x \in A(S)$ with $x,\exp_S(x) \in R$, but $x \notin A(R)$.

The category of E-rings is defined just by functions and equations, so there is a notion of a free E-ring. We write $\Z[X]^E$ for the free E-ring on a set of generators $X$, and call it the E-ring of \emph{exponential polynomials} in indeterminates $X$. Similarly for any E-ring $R$ we can consider the free E-ring extension of $R$ on a set of generators $X$, written $R[X]^E$, and call it the E-ring of exponential polynomials over $R$ (or with coefficients in $R$). See \cite{Macintyre96} for an explicit construction.

\section{Exponential algebraicity}

Exponential algebraic closure is the analogue in E-domains of the
notion of (relative) algebraic closure in pure domains. In a domain
$R$, an element $a$ is algebraic over a subring $B$ iff it satisfies a
non-trivial polynomial over $B$. In the E-domain context, we need a
slightly more complicated definition.

\begin{defn}\label{ecl defn}
  Let $R$ be an E-domain. A \emph{Khovanskii system} (of equations and
  inequations) consists of, for some $n \in \N$, exponential
  polynomials $f_1,\ldots,f_n \in R[X_1,\ldots,X_n]^E$, with equations
  \[f_i(x_1,\ldots,x_n) = 0 \quad \mbox{for } i=1,\ldots,n\] and the
  inequation
\[\begin{vmatrix}
  \frac{\partial f_1}{\partial X_1} & \cdots &\frac{\partial
    f_1}{\partial X_n}\\
  \vdots & \ddots & \vdots \\
  \frac{\partial f_n}{\partial X_1} & \cdots &\frac{\partial
    f_n}{\partial X_n} \end{vmatrix} (x_1,\ldots,x_n) \neq 0.\]
\end{defn}
In the analytic context of $\Rexp$ or $\Cexp$, the $f_i$ are analytic
functions, and the non-vanishing of the Jacobian means that $\bar{x}$
is an \emph{isolated} zero of the system of equations
$\bar{f}(\bar{x}) = 0$. However, the notion of a Khovanskii system is
purely algebraic, so we do not need any topology to make sense of it.

\begin{defn}
  If $B$ is an E-subring of $R$, define $a \in \ecl^R(B)$ iff
  there are $n \in \N$, $a_1,\ldots,a_n \in R$, and $f_1,\ldots,f_n
  \in B[X_1,\ldots,X_n]^E$ such that $a = a_1$ and
  $(a_1,\ldots,a_n)$ is a solution to the Khovanskii system given by
  the $f_i$.  If $C \subs F$ is any subset, let $\hat{C}$ be the
  E-subring of $R$ generated by $C$, and define $\ecl^R(C) =
  \ecl^R(\hat{C})$.

  We say that $\ecl^R(C)$ is the \emph{exponential algebraic closure} of $C$ in $R$. If $a \in \ecl^R(C)$ we say that $a$ is \emph{exponentially algebraic} over $C$ in $R$, and otherwise that it is \emph{exponentially transcendental} over $C$ in $R$. When $R$ is a partial E-domain, the same definition works but we must be careful only to apply exponential polynomial functions where they are defined.
\end{defn}
\begin{lemma}\label{ecl closure}
  If $R$ is a partial E-domain then $\ecl^R$ is a closure operator with finite character. That is, for any subsets $C,B$ of $R$ we have
\begin{itemize}
 \item $C \subs \ecl^R(C)$
 \item $B \subs C \implies \ecl^R(B) \subs \ecl^R(C)$
 \item $\ecl^R(\ecl^R(C)) = \ecl^R(C)$
 \item $\ecl^R(C) = \bigcup \class{\ecl^R(C_0)}{C_0 \mbox{ is a finite subset of } C}$.
\end{itemize}
Furthermore, the closure of any subset is an E-subring of $R$, and, if $R$ is a field, it is an E-subfield.
\end{lemma}
\begin{proof}
  A straightforward exercise.
\end{proof}

It is also easy to see that if $R \subs S$ are E-domains and $C \subs
R$ then $\ecl^R(C) \subs \ecl^S(C) \cap R$. However, unlike in the
case of algebraic closure, this inclusion may be strict.

\begin{remark}\label{CCP}
  On $\Rexp$ or $\Cexp$, there can only be countably many isolated zeros of a system of equations, so it follows that there are only countably many exponentially algebraic numbers. It is, of course, a difficult problem to show that any number is even transcendental, and as far as I know there are no real or complex numbers which are known to be exponentially transcendental. It seems likely that the Liouville numbers are all exponentially transcendental, but that may be difficult to prove.
\end{remark}

\section{Derivations and differentials}

  Derivations play an important role in transcendence theory for pure
  fields. The analogous notion for exponential fields was first
  exploited by Wilkie. Here we define exponential derivations and
  differentials, in analogy with the theory of differentials in
  commutative algebra.

\begin{defn}
  Let $R$ be a partial E-ring, and $M$ an $R$-module. (There is no
  exponential structure on $M$; it is just a module in the usual
  sense.) A \emph{derivation} from $R$ to $M$ is a map $\ra{R}{\D}{M}$
  such that for each $a,b \in R$,
  \begin{itemize}
  \item $\D (a+b) = \D a + \D b$, and
  \item $\D (ab) = a \D b + b \D a$
  \end{itemize}
  It is an \emph{exponential derivation} or \emph{E-derivation} iff also for each $a \in A(R)$ we have $\D (\exp(a)) = \exp(a) \D a$. 

  Write $\Der(R, M)$ for the set of all derivations from $R$ to $M$, and $\EDer(R,M)$ for the set of all E-derivations from $R$ to $M$. For any subset $C$ of $R$, we write $\Der(R/C,M)$ and $\EDer(R/C,M)$ for the sets of derivations (E-derivations) which vanish on $C$. It is easy to see that these are $R$-modules. 
\end{defn}

We have the \emph{universal derivation} $\ra{R}{d}{\Omega(R/C)}$, where $\Omega(R/C)$ is the $R$-module generated by symbols $\class{dr}{r \in R}$, subject only to the relations given by $d$ being a derivation and the relations $dc=0$ for each $c \in C$. Similarly there is a universal E-derivation, $\ra{R}{d}{\Xi(R/C)}$, where $\Xi(R/C)$ is the quotient of $\Omega(R/C)$ defined by the extra relations of an E-derivation. The universal property is that if $\ra{R}{\D}{M}$ is any E-derivation vanishing on $C$ then there is a unique $R$-linear map $\D^*$ such that
\[\begin{diagram}[height=2em,width=2em]
   R & \rTo^d & \Xi(R/C)\\
     & \rdTo_\D & \dTo>{\D^*}\\
     & & M
  \end{diagram}\]
commutes.

An important special case is when $M = R$. In this case, we write $\Der(R/C)$ for $\Der(R/C,R)$ and $\EDer(R/C)$ for $\EDer(R/C,R)$. When $C = \emptyset$ we also write $\Der(R)$ and $\EDer(R)$.

Unlike in the case of pure fields, it is not easy to see what
the derivations on a given E-field are. The reason for this is that a
derivation on an E-field $F_1$ may not extend to an extension E-field
$F_2 \sups F_1$. This phenomenon also occurs for pure fields, but only
in positive characteristic and only in one way, when giving new
$p^{\mathrm{th}}$ roots.
\begin{example}
  Consider the extension of pure fields $\mathbb{F}_p(t) \subs
  \mathbb{F}_p(s)$, where $t = s^p$. On $\mathbb{F}_p(t)$ we have the
  derivation $\frac{\D}{\D t}$. But if $\D$ is any derivation on
  $\mathbb{F}_p(s)$ then $\D t = \D(s^p) = p s^{p-1}\D s = 0$. So $\D$
  is not an extension of $\frac{\D}{\D t}$.
\end{example}

In pure fields of characteristic zero, if $F_1 \subs F_2$ then $\dim
\Der(F_2/F_1) = \td(F_2/F_1)$. Furthermore, $a \in \acl(F_1)$ iff
$\td(F_1(a)/F_1) = 0$ iff every derivation on $F_2$ which vanishes on
$F_1$ also vanishes at $a$. By analogy, we define a closure operator
$\cl^R$ on an E-domain $R$ as follows.
\begin{defn}
  For $R$ a partial E-domain, $C \subs R$ and $a \in R$, define $a\in \cl^R(C)$ iff  for every $\D \in
  \EDer(R/C)$, $\D a = 0$.
\end{defn}
By the universal property of $\Xi(R/C)$, $a\in \cl^R(C)$ iff $dx=0$ in $\Xi(R/C)$.

\begin{lemma}
  The operator $\cl^R$ is a closure operator satisfying the exchange property. Furthermore the closure of any subset is an E-subring, and, if $R$ is a field, an E-subfield.
\end{lemma}
\begin{proof}
  It is immediate that $C \subs \cl^R(C)$, if $C_1 \subs C_2$ then
  $\cl^R(C_1) \subs \cl^R(C_2)$, and that $\cl^R(\cl^R(C)) =
  \cl^R(C)$. It is also immediate that $\cl^R(C)$ is closed under the
  E-ring operations and under taking multiplicative inverses. For
  exchange, suppose that $a \in \cl^R(Cb)$ but $b \notin
  \cl^R(Ca)$. Then there is an E-derivation $\D$ which vanishes on $C$
  such that $\D a = 0$ and $\D b = 1$. Let $\D' \in \EDer(R/C)$, and
  let $\D'' = \D' - (\D' b) \D$. Then $\D' a = \D'' a$, but $\D'' b =
  0$ and $a \in \cl^R(Cb)$, so $\D'' a = 0$. Hence $\D' a = 0$, and so
  $a \in \cl^R(C)$.
\end{proof}
Wilkie explicitly builds finite character into the definition of $\cl^R$ to give a pregeometry. In fact this is not necessary, as finite character holds already.

\begin{prop}\label{finite character}
  Suppose $R$ is a partial E-domain, $C \subs R$ and $a \in \cl^R(C)$. Then there is a finite subset $C_0$ of $C$ and a finitely generated partial E-subring $R_0$ of $R$ such that $a \in \cl^{R_0}(C_0)$. Furthermore, $\cl^R$ has finite character, and is a pregeometry.
\end{prop}
\begin{proof}
  We have $da = 0$ in $\Xi(R/C)$. We use a simple compactness argument. Let $L$ be a formal language with a constant symbol for each finite sum $\sum r_i ds_i$ with the $r_i, s_i \in R$. Let $T$ be the $L$-theory consisting of all instances of the axioms saying that these symbols represent elements of the $R$-module $\Xi(R/C)$ that is, the axioms of an $R$-module, the axioms saying that $d$ is an E-derivation, and the axioms $dc = 0$ for each $c \in C$. Then $T \proves da = 0$. Hence by compactness there is a finite subtheory $T_0$ of $T$ such that $T_0 \proves da = 0$. Let $C_0$ be the subset of $C$ consisting of those $c$ such that the axiom $dc=0$ appears in $T_0$. Let $R_0$ be the partial E-subring of $R$ generated by all the $r \in R$ which occur in some axiom of $T_0$. Then we must have $da = 0$ in $\Xi(R_0/C_0)$, and also in $\Xi(R/C_0)$. Thus $a \in \cl^{R_0}(C_0)$, and $a \in \cl^{R}(C_0)$, which gives finite character of $\cl^R$. We have shown that $\cl^R$ satisfies the other axioms of a pregeometry.
\end{proof}

We now begin to relate our closure operators $\cl^R$ and $\ecl^R$.
\begin{lemma}\label{Xi char}
  $\Xi(R/C)$ can also be characterized as the $R$-module generated by symbols $\class{dr}{r \in R}$ subject to the relations 
  \[\sum_{i=1}^n \frac{\partial f}{\partial X_i}(\bar{r}) dr_i = 0 \qquad \qquad (*)\]
  for each $f \in C[\bar{X}]^E$ and tuple $\bar{r}$ from $R$ such that $f(\bar{r}) = 0$.
\end{lemma}
\begin{proof}
  The relation $d(x+y) = dx + dy$ comes from $f = X_1+X_2-X_3$, and similarly for the other basic relations axiomatizing E-derivations. Conversely, the relations $(*)$ follow from the axioms of E-derivations by the chain rule.
\end{proof}

\begin{prop}\label{ecl subs cl}
Let $R$ be a partial E-domain and $C$ a subset of $R$. Then $\ecl^R(C) \subs \cl^R(C)$.
\end{prop}
\begin{proof}
  Both closures of $C$ are E-subrings of $R$, so we may assume that $C$ is an E-subring.
  Suppose $a_1,\ldots,a_n \in \ecl^R(C)$, as witnessed by being a
  solution to the Khovanskii system formed by $f_1,\ldots,f_n \in
  C[X_1,\ldots,X_n]^E$. Suppose $\D \in \Der(R/C)$, and let $J$ be
  the Jacobian matrix $J = 
  \begin{pmatrix} 
    \frac{\partial f_1}{\partial X_1} & \cdots & \frac{\partial
      f_1}{\partial X_n} \\
    \vdots & \ddots & \vdots \\
    \frac{\partial f_n}{\partial X_1} & \cdots & \frac{\partial
      f_n}{\partial X_n}
  \end{pmatrix}(\bar{a})$. Then by lemma~\ref{Xi char},
  $J \begin{pmatrix} \D a_1 \\ \vdots \\ \D a_n \end{pmatrix} = 0$.
  Since $\bar{a}$ solves the Khovanskii system, the determinant $|J| \neq 0$, so $J$ has an inverse with coefficients in the field of fractions of $R$. Clearing denominators, for some nonzero $r \in R$ the matrix $rJ^{-1}$ has coefficients in $R$. Then $r \begin{pmatrix} \D a_1 \\ \vdots \\ \D a_n \end{pmatrix} = 0$ and hence $\D a_i = 0$ for each $i$, as $R$ is a domain. So each $a_i$ lies in $\cl^R(C)$.
\end{proof}

It will be useful to have a stronger form of lemma~\ref{Xi char} for finitely generated extensions of partial E-fields, where we consider only relations between the chosen generators.
\begin{lemma}\label{Xi char2}
  Suppose $C \into F$ is an inclusion of partial E-fields, that $a_1,\ldots,a_n$ is a $\Q$-linear basis for $A(F)$ over $A(C)$, and that $F$ is generated as a field by $A(F) \cup \exp(A(F))$. Then $\Xi(F/C)$ is the $F$-vector space generated by $da_1,\ldots,da_n$ subject to the relations 
  \[\sum_{i=1}^n \frac{\partial f}{\partial X_i}(\bar{a}) da_i = 0 \qquad \qquad (*)\]
 for each $f \in C[\bar{X}]^E$ such that $f(\bar{a}) = 0$.
\end{lemma}
\begin{proof}
  The differentials $d e^{a_i}$ satisfy $d e^{a_i} = e^{a_i} da_i$, so are in the span of the $da_i$. $F$ is algebraic over $C(\bar{a},e^{\bar{a}})$, so these differentials span $\Omega(F/C)$, hence they certainly span $\Xi(F/C)$. We must show that the basic relations axiomatizing E-derivations follow from the relations $(*)$. All of the exponential relations $de^b = e^b db$ follow from those for the $a_i$ by $\Q$-linearity. We are left with the algebraic relations between elements of $F$. Suppose 
\[ p\left(\frac{f_1(\bar{a})}{g_1(\bar{a})},\ldots,\frac{f_m(\bar{a})}{g_m(\bar{a})}\right) = 0 
\qquad\qquad (\dagger)\]
with $p \in C[Y_1,\ldots,Y_m]$ and the $f_i,g_i$ exponential polynomials, with $g_i(\bar{a}) \neq 0$. Clearing the denominators, we get an exponential polynomial $h(\bar{X})$ such that
($\dagger$) is equivalent to $h(\bar{a}) = 0$. So 
\[d\left[  p(f_1(\bar{a})/g_1(\bar{a}),\ldots,f_m(\bar{a})/g_m(\bar{a}))\right]  = 0 \iff d\left[ h(\bar{a})\right]  = 0\]
but this is iff $\sum_{i=1}^n \frac{\partial h}{\partial X_i}(\bar{a}) da_i = 0$ which is of the form $(*)$. So the relations of the form $(*)$ are enough to characterize $\Xi(F/C)$.
\end{proof}

\section{Strong extensions}

We need the following theorem of J.~Ax.
\begin{theorem}[{\cite[theorem~3]{Ax71}}]\label{Ax theorem}
  Let $F$ be a field of characteristic 0, let $\Delta$ be a set of
  derivations on $F$, and let $C = \bigcap_{\D \in
    \Delta}\ker \D$ be the field of constants. Suppose
  $x_1,\ldots,x_n, y_1,\ldots,y_n \in F$ satisfy $\D y_i = y_i \D x_i$
  for each $i = 1,\ldots, n$ and each $\D \in \Delta$. Then
  \[\td(\bar{x},\bar{y}/C) \ge \ldim_\Q(\bar{x}/C) + \rk
  \begin{pmatrix}\D x_i \end{pmatrix}_{\D \in \Delta, i=1,\ldots,n}\]
\end{theorem}

\begin{cor}\label{Ax corollary}
Let $F$  be an E-field, and suppose $C \subs F$ is $\cl^F$-closed. Let
$x_1,\ldots,x_n \in F$. Then
\[\td(\bar{x},\exp(\bar{x})/C) - \ldim_\Q(\bar{x}/C) \ge
\dim^F(\bar{x}/C)\] where $\dim^F(\bar{x}/C)$ is the dimension in the sense
of the pregeometry $\cl^F$.
\end{cor}
\begin{proof}
  Taking $\Delta = \EDer(F/C)$ and $y_i = \exp(x_i)$, all the differential
  equations $\D y_i = y_i \D x_i$ for $\D \in \Delta$ are satisfied. Also $C = \bigcap_{\D \in \Delta}\ker \D$ because $C$ is $\cl^F$-closed. We can find $x_{i_1},\ldots,x_{i_m}$ among
  the $x_i$, where $m =
  \dim^F(\bar{x}/C)$, and derivations $\D_1,\ldots,\D_m \in \Delta$
  such that $\D_j x_{i_k} = \delta_{jk}$, the Kronecker delta. Thus $\rk
  \begin{pmatrix}\D x_i \end{pmatrix}_{\D \in \Delta, i=1,\ldots,n} =
  m$. Apply Ax's theorem.
\end{proof}

Now let $R$ be any partial E-domain.  For any tuple $\bar{x}$ and subset $B$ of $A(R)$, we define
\[\delta(\bar{x}/B) = \td(\bar{x},\exp(\bar{x})/B,\exp(B)) -
\ldim_\Q(\bar{x}/B) \] which, following Hrushovski, we call the \emph{predimension} of
$\bar{x}$ over $B$.  Note that if $B = \bar{b}$ is finite then we have
the useful addition formula $\delta(\bar{x}/\bar{b}) =
\delta(\bar{x}\bar{b}/0) - \delta(\bar{b}/0)$.

\begin{defn}
  We say an embedding $R_1 \into R_2$ of partial E-domains is
  \emph{strong}, and write $R_1 \strong R_2$, iff for every tuple
  $\bar{x}$ from $A(R_2)$, we have $\delta(\bar{x}/A(R_1)) \ge 0$.

  More generally, if $B$ is any subset of $A(R)$ for a partial
  E-domain $R$, we say that $B$ is strong in $R$, and write $B \strong
  R$, iff for every tuple $\bar{x}$ from $A(R)$, we have
  $\delta(\bar{x}/B) \ge 0$.
\end{defn}

Not all E-field extensions are strong. For example, $\Rexp \subs
\Cexp$ is not strong, since $\delta(i/\R) = \td(i, e^i/\R) -
\ldim_\Q(i/\R) = 0 - 1 = -1$. This example can be generalized to show
that any proper algebraic extension, or even one of finite
transcendence degree, cannot be strong.
\begin{lemma}\label{no finite td}
  If $F_0 \strong F$ is a strong extension of total E-fields and
  $\td(F/F_0)$ is finite, then $F = F_0$.
\end{lemma}
\begin{proof}
  Suppose $F$ is a proper, strong extension of $F_0$. Choose a
  $\Q$-linear basis $\class{b_i}{i \in I}$ for $F$ over $F_0$. Then
  \[\td(F/F_0) \ge \td\left(\class{b_i, e^{b_i}}{i \in I}/F_0\right) \ge \card{I}\]
  which means that $I$ is finite. But then $I = \emptyset$ or $F$ is a
  finite extension of $\Q$, in particular algebraic, so $\td(F/F_0) =
  0$, in which case $I = \emptyset$ anyway. Thus $F = F_0$.
\end{proof}
However, if we allow partial exponential fields, every strong
extension can be split up into a chain of strong extensions of finite
transcendence degree. To show this we need some basic properties of strong
extensions, which are left as a straightforward exercise.
\begin{lemma}\label{AEC axioms}
For ordinals $\alpha$, let $R_\alpha$ be partial E-domains.
  \begin{enumerate}
  \item[(i)] The identity $R_1 \into R_1$ is strong.
  \item[(ii)]  If $R_1 \strong R_2$ and $R_2 \strong R_3$ then $R_1 \strong
    R_3$. (That is, the composite of strong extensions is strong.)
  \item[(iii)] Suppose $\lambda$ is an ordinal, $(R_\alpha)_{\alpha <
      \lambda}$ is a $\lambda$-chain of strong extensions (that is for
    each $\alpha \le \beta < \lambda$ there is a strong extension
    $f_{\alpha,\beta}: R_\alpha \strong R_\beta$ and for all $\alpha
    \le \beta \le \gamma$, $f_{\beta,\gamma} \circ f_{\alpha,\beta} =
    f_{\alpha,\gamma}$ and $f_{\alpha,\alpha}$ is the identity on
    $R_\alpha$), and $R$ is the union of the chain. Then $R_\alpha
    \strong R$ for each $\alpha$.
  \item[(iv)] Suppose $(R_\alpha)_{\alpha < \lambda}$ is a $\lambda$-chain of
    strong extensions with union $R$, and that $R_\alpha \strong S$
    for each $\alpha$. Then $R \strong S$.
  \end{enumerate}\qed
\end{lemma}

\begin{prop}\label{chain prop}
  Suppose $F_0 \strong F$ is a strong extension of partial E-fields, and $F_0,F$ are \emph{\egg}, that is, they are generated as fields by the graphs of their exponential maps, $A(F) \cup \exp(A(F))$, and similarly for $F_0$. Then for
  some ordinal $\lambda$ there is a chain $(F_\alpha)_{\alpha \le
    \lambda}$ of partial E-domains such that for
  all ordinals $0 \le \alpha \le \beta \le \lambda$,
  \begin{enumerate}
  \item $F = F_\lambda$
  \item $F_\alpha$ is \egg
  \item For limit $\beta$, $F_\beta = \bigcup_{\alpha <
      \beta}F_\alpha$ 
  \item $\td(F_{\beta+1}/F_\beta)$ is finite
  \item $F_\alpha \strong F_\beta$
  \end{enumerate}
\end{prop}
\begin{proof}
  Let $\lambda$ be the initial ordinal of cardinality $\card{A(F)}$, and list
  $A(F)$ as $(r_\alpha)_{\alpha<\lambda}$.
  We inductively construct $F_\beta$ satisfying (1)---(5) and such that $r_\beta \in F_{\beta+1}$ and $F_\beta \strong F$.

  At a limit stage $\beta$, define $A(F_\beta) = \bigcup_{\alpha < \beta}A(F_\alpha)$. Take $F_\beta$ to be the partial E-subfield of $F$ generated by $A(F_\beta)$, so (2) and (3) hold. (5) holds by part (iii) of lemma~\ref{AEC axioms}, and $F_\beta \strong F$ by part (iv) of lemma~\ref{AEC axioms}.

  For a successor $F_{\beta+1}$, if $r_\beta \in A(F_\beta)$, take $F_{\beta+1} =  F_\alpha$. Otherwise, by induction $F_\beta \strong F$, so for any finite tuple $\bar{x}$ from $A(F)$ we have $\delta(\bar{x}/F_\beta) \ge 0$. Choose a tuple $\bar{x}$ containing $r_\beta$ such that $\delta(\bar{x}/R_\beta)$ is minimal. Let $A(F_{\beta+1})$ be the $\Q$-subspace of $A(F)$ generated by $A(F_\beta)$ and $\bar{x}$, and take $F_{\beta+1}$ to be the partial E-subfield of $F$ generated by $A(F_\beta)$. By the minimality of $\delta(\bar{x}/F_\beta)$, $F_{\beta+1} \strong F$. For any $\alpha \le \beta$, since $F_\alpha \strong F$, it follows that $F_\alpha \strong F_{\beta+1}$. Also $\td(F_{\beta+1}/F_\beta) \le 2 \card{\bar{x}}$ which is finite, so (4) holds. Finally, $\bigcup_{\alpha<\lambda}A(F_\alpha) = A(F)$, so $F = F_\lambda$.
\end{proof}

\section{Extending derivations}

Let $F_0 \subs F$ be an extension of (pure) fields, and let $\D \in \Der(F_0)$. There are spaces of differentials $\Omega(F)$ and $\Omega(F/F_0)$ appropriate for considering all derivations on $F$ and those which vanish on $F_0$. We construct an intermediate space of differentials appropriate for considering extensions of $\D$ to $F$. 

\begin{defn}
 Let $\Omega(F/\D)$ be the quotient of $\Omega(F)$ by the relations $\sum a_i db_i = 0$ for those $a_i,b_i \in F_0$ such that $\sum a_i \D b_i = 0$.
\end{defn}
We naturally have quotient maps
\[\Omega(F) \rOnto \Omega(F/\D) \rOnto \Omega(F/F_0).\]
\begin{lemma}
 Let $\Der(F/\D) = \class{\eta \in \Der(F)}{(\exists \lambda \in F) \eta\restrict{F_0} = \lambda\D}$. Then $\Der(F/\D)$ is the dual space of $\Omega(F/\D)$.
\end{lemma}
\begin{proof}
  Suppose $\eta \in \Der(F/\D)$. Then for each relation $\sum a_i \D b_i = 0$ we have $\sum a_i \eta b_i = \lambda \sum a_i \D b_i = 0$, so $\eta$ factors as
\[F \rTo^d \Omega(F/\D) \rTo^{\eta^*} F\]
for some $F$-linear map $\eta^*$.
Now if $\Omega(F/\D) \rTo^{\eta^*} F$ is any $F$-linear map, define $\eta = \eta^* \circ d$. Then $\eta \in \Der(F)$ and we must show $\eta \in \Der(F/\D)$. If $\D b = 0$ for some $b \in F_0$, then the relation $db=0$ holds in $\Omega(F/\D)$ and so $\eta b = 0$. If this holds for all $b \in F_0$ we are done. Otherwise choose $b_0 \in F_0$ such that $\D b_0 \neq 0$ and let $\lambda = \eta b_0/\D b_0$. Let $b \in F_0$, and write $b' = \D b$ and $b_0' = \D b_0$. Then $b_0' \D b - b' \D b_0 = 0$, so $b_0' db - b'db_0 = 0$ in $\Omega(F/\D)$, so $b_0'\eta b - b' \eta b_0 = 0$, that is, $\eta b = \lambda \D b$. Hence $\eta\restrict{F_0} = \lambda \D$ and so $\eta \in \Der(F/\D)$.
\end{proof}

\begin{theorem}\label{derivations extend}
  Suppose $F_0 \strong F$ is a strong extension of partial E-fields and $F_0$ is \egg. Then every E-derivation on $F_0$ extends to $F$.
\end{theorem}
%The proof is similar to the proof of
%\cite[theorem~3.10]{TEDESV}, which is concerned with solving exponential
%differential equations, although in the more general context of the
%exponential maps of semiabelian varieties. 

\begin{proof}%[Proof of theorem~\ref{derivations extend}]
  Let $F'$ be the partial E-subfield of $F$ generated by the graph of exponentiation of $F$. Then every E-derivation on $F'$ extends to $F$, as only the field operations must be respected and the characteristic is zero. So we may assume $F=F'$. Now by proposition~\ref{chain prop} it is enough to prove the theorem for extensions of \egg\ partial E-fields $F_1 \strong F_2$ such that $\td(F_2/F_1)$ is finite. Let $\D$ be an E-derivation on $F_1$. Let $\EDer(F_2/\D) = \Der(F_2/\D) \cap \EDer(F_2)$.

  Let $a_1,\ldots,a_n$ be a $\Q$-basis for $A(F_2)$ over $A(F_1)$, and let $\omega_i = \frac{de^{a_i}}{e^{a_i}} - da_i \in \Omega(F_2)$. Let $\hat{\omega}_i$ be the image of $\omega_i$ in $\Omega(F_2/F_1)$ under the natural quotient map
$\Omega(F_2) \rOnto \Omega(F_2/F_1)$.

  We use the following intermediate step in the proof of Ax's theorem (\ref{Ax
  theorem} of this paper). Although this statement is not isolated in Ax's paper, it can be obtained from his proof. It is also the special case of proposition~3.7 of \cite{TEDESV} where the group $S$ is $\gm^n$.
  \begin{fact}\label{Ax fact}
    If the differentials $\hat{\omega}_1,\ldots,\hat{\omega}_n$ are $F_2$-linearly
    dependent in $\Omega(F_2/F_1)$ then there is a non-zero $\Z$-linear combination $b = \sum_{i=1}^n m_i a_i$ such that $b$ and $e^b$ are both algebraic over $F_1$.
  \end{fact}
  So if the $\hat{\omega}_i$ are $F_2$-linearly dependent, then for some such $b$ we have 
  \begin{eqnarray*}
    \delta(b/F_1) &=& \td(b,e^b/A(F_1) \cup \exp(A(F_1))) - \ldim_\Q(b/A(F_1))\\
	  	  &=& \td(b,e^b/F_1) - \ldim_\Q(b/A(F_1))\\
	 	  &=& 0 - 1 < 0
  \end{eqnarray*}
  which contradicts $F_1 \strong F_2$. Thus the $\hat{\omega}_i$ are $F_2$-linearly independent in $\Omega(F_2/F_1)$.

  Let $\Lambda$ be the $F_2$-subspace of $\Omega(F_2)$ generated by $\omega_1,\ldots,\omega_n$. 
  The space of derivations $\Der(F_2)$ is the dual space of $\Omega(F_2)$, so we can consider the annihilator of $\Lambda$ in it. By definition of $\Lambda$, $\EDer(F_2/F_1) = \Der(F_2/F_1) \cap \Ann(\Lambda)$ and $\EDer(F_2/\D) = \Der(F_2/\D) \cap \Ann(\Lambda)$.
We have shown that the image of $\Lambda$ in $\Omega(F_2/F_1)$ has dimension $n$, and hence the image of $\Lambda$ in $\Omega(F_2/\D)$ also has dimension $n$. The subspaces $\Der(F_2/F_1)$ and $\Der(F_2/\D)$ of $\Der(F_2)$ are dual to the quotients $\Omega(F_2/F_1)$ and $\Omega(F_2/\D)$ of $\Omega(F_2)$, and hence $\Ann(\Lambda)$ has codimension $n$ in $\Der(F_2/F_1)$, and also in $\Der(F_2/\D)$.

If $\D = 0$ the result is trivial. Otherwise, $\dim \Der(F_2/\D) = \dim \Der(F_2/F_1) + 1$, so $\dim \EDer(F_2/\D) = \dim \EDer(F_2/F_1) + 1$. Thus there is $\eta \in \EDer(F_2/\D) \minus \EDer(F_2/F_1)$.
Then $\eta\restrict{F_1} = \lambda \D$ for some non-zero $\lambda$. Let $\eta' = \lambda^{-1}\eta$. Then $\eta'$ extends $\D$ to $F_2$.
\end{proof}

If $F$ is an E-field with $\EDer(F) = \{0\}$, then of course this zero
derivation extends to any E-field extension. Not all E-field
extensions of $F$ are strong, so the converse to the theorem is false.

\section{Proofs of the main theorems}

\begin{prop}
  If $F$ is a partial E-field and $C$ is a subset of $F$ then $\cl^F(C) \subs \ecl^F(C)$.
\end{prop}
\begin{proof}
  Suppose $a \in \cl^F(C)$. We may assume that $F$ is \egg, that $C$ is finite, and that $C \subs A(F)$.
Now $a$ must be algebraic over the graph of exponentiation, and we know that both $\cl^F(C)$ and $\ecl^F(C)$ are relatively algebraically closed in $F$, so it is enough to prove the proposition for $a$ in the graph of exponentiation. Also, replacing $a$ by some $a'$ with $\exp(a') = a$ if necessary, we may assume that $a \in A(F)$.

By proposition~\ref{finite character}, there is $F_1$, an \egg\ partial E-subfield of $F$ such that $A(F_1)$ contains $C$ and $a$ and is finitely generated, and such that $a \in \cl^{F_1}(C)$. Choose such an $F_1$ with $\ldim_\Q(A(F_1)/C)$ minimal.

Let $F_2 = \cl^{F_1}(C)$. We claim that $F_2 = F_1$. Certainly $a \in F_2$, so if $F_2 \neq F_1$ then by minimality of $F_1$ we have $a \notin \cl^{F_2}(C)$. So there is an E-derivation $\D \in \EDer(F_2/C)$ which does not extend to $F_1$. Then, by theorem~\ref{derivations extend}, $F_2 \nstrong F_1$, but that contradicts corollary~\ref{Ax corollary} since $F_2$ is $\cl^{F_1}$-closed in $F_1$. Hence $F_2 = F_1$.

  By lemma~\ref{Xi char2}, $\Xi(F_1/C)$ is generated by $da_1,\ldots,da_n$, subject to the relations 
  \[\sum_{i=1}^n \frac{\partial f}{\partial X_i}(\bar{a}) da_i = 0 \]
  for $f \in C[\bar{X}]^E$ such that $f(\bar{a}) = 0$ in $F_1$.
  Since $F_1 = \cl^{F_1}(C)$, we have $\Xi(F_1/C) = 0$. Hence we can choose $f_1,\ldots,f_n$ such that the matrix $J = \begin{pmatrix} \frac{\partial f_j}{\partial X_i}(\bar{a})\end{pmatrix}_{i,j=1}^n$ has rank $n$, that is, is non-singular. Thus $a \in \ecl^{F_1}(C)$.

Now $F_1 \subs F$, so by the remark after lemma~\ref{ecl closure}, $\ecl^{F_1}(C) \subs \ecl^F(C)$. Hence $a \in \ecl^F(C)$ as required.
\end{proof}
Together with proposition~\ref{ecl subs cl}, that completes the proof that $\ecl^F = \cl^F$ for any partial E-field $F$, and it follows that $\ecl^F$ is a pregeometry. Theorem~\ref{ecl is pregeometry} is established, and theorem~\ref{Schanuel property} follows from corollary~\ref{Ax corollary}.

\begin{proof}[Proof of theorem~\ref{dim characterization}]
  Let $C = \ecl^F(\emptyset)$ and choose $\bar{y}$ such that $r \leteq \delta(\bar{x}\bar{y}/C)$ is minimal. Let $F_0$ be the \egg\ partial E-field extension of $C$ with $A(F_0)$ generated by $\bar{x}\bar{y}$ over $C$. Using the notation from the proof of theorem~\ref{derivations extend}, 
$\EDer(F_0/C) = \Der(F_0/C) \cap \Ann(\Lambda)$,
but $\Ann(\Lambda)$ has codimension $\ldim_\Q(\bar{x}\bar{y}/C)$ in $\Der(F_0/C)$ by fact~\ref{Ax fact}, hence
\[\ldim_{F_0}\EDer(F_0/C) = \td(F_0/C) - \ldim_\Q(\bar{x}\bar{y}/C) = \delta(\bar{x}\bar{y}/C) = r.\]
Since $\bar{y}$ is chosen with minimal $\delta$ we have $F_0 \strong F$, so, by theorem~\ref{derivations extend}, these E-derivations all extend to $F$. Hence $\dim^F(\bar{x}) \ge r$. Then, by theorem~\ref{Schanuel property}, $\dim^F(\bar{x}) = r$ as required.
\end{proof}

To prove theorem~\ref{SC counterexamples}, we give a more general result.
\begin{prop}
  In any partial E-field $F$, if $\bar{a}$ is an essential counterexample to the Schanuel property then $\bar{a}$ is contained in $\ecl^F(\emptyset)$.
\end{prop}
\begin{proof}
  Let $\bar{a}$ be a tuple from $F$, write $\tuple{\bar{a}}_\Q$ for its $\Q$-linear span, let $B = \tuple{\bar{a}}_\Q \cap \ecl^F(\emptyset)$, and suppose that $\bar{a} \not\subs \ecl^F(\emptyset)$, so $\tuple{\bar{a}}_\Q \neq B$. Then
\[\td(\bar{a},\exp(\bar{a})/B, \exp(B)) \ge \td(\bar{a},\exp(\bar{a})/\ecl^F(\emptyset)) \]
 and $\ldim_\Q(\bar{a}/B) = \ldim_\Q(\bar{a}/\ecl^F(\emptyset))$. So
\[\delta(\bar{a}/B) \ge \delta(\bar{a}/\ecl^F(\emptyset)) \ge \dim^F(\bar{a}) \ge 1.\]
and thus $\delta(B) = \delta(\bar{a}) - \delta(\bar{a}/B) < \delta(\bar{a})$, hence $\bar{a}$ is not an essential counterexample.
\end{proof}
Now, by remark~\ref{CCP}, $\ecl^\C(\emptyset)$ is countable, which proves theorem~\ref{SC counterexamples}.

%\bibliographystyle{../pjk1}
%\bibliography{../papers,../books,../drafts}

\end{document}